\documentclass[11pt]{amsart}
 \usepackage{ucs}

 \usepackage{lmodern}
 \usepackage[T1]{fontenc}
 \usepackage{graphicx}

\usepackage{amssymb, amsmath}
\usepackage[usenames,dvipsnames]{xcolor}

\usepackage{color,ulem,textcomp}

\numberwithin{equation}{section}
\usepackage{enumerate} 
\usepackage{graphicx}
\usepackage{cancel}

\def\al{\alpha}

\newcommand{\ds}{\displaystyle}

\def\cB{{\mathcal B}}
\def\cC{{\mathcal C}}

\def\v{{\rm v}}

\def\N {\mathbf{N}}

\def\T {\mathbb{T}}

\def\Z {\mathbf{Z}}
\def\R {\mathbf{R}}

\def\cB {\mathcal{B}}
\def\cC {\mathcal{C}}

\def\eps {{\varepsilon}}
\def\e {{\varepsilon}}

\newcommand{\ba}{\begin{aligned}}
\newcommand{\ea}{\end{aligned}}

\newcommand{\be}{\begin{equation}}
\newcommand{\ee}{\end{equation}}

\def\virgp{\raise 2pt\hbox{,}}

\newcommand{\beq}{\begin{equation}}
\newcommand{\eeq}{\end{equation}}
\newcommand{\ben}{\begin{eqnarray}}
\newcommand{\een}{\end{eqnarray}}
\newcommand{\beno}{\begin{eqnarray*}}
\newcommand{\eeno}{\end{eqnarray*}}
\def\sumetage#1#2{
\sum_{\scriptstyle {#1}\atop\scriptstyle {#2}} }

\def\supetage#1#2{\sup_{\substack{{#1}\\{#2}}}}

\def\PP{{\mathop{\mathbb P\kern 0pt}\nolimits}}
\def\Q{{\mathop{\mathbb Q\kern 0pt}\nolimits}}
\def\R{{\mathop{\mathbb R\kern 0pt}\nolimits}}
\def\SS{{\mathop{\mathbb S\kern 0pt}\nolimits}}
\def\ZZ{{\mathop{\mathbb Z\kern 0pt}\nolimits}}
\def\TT{{\mathop{\mathbb T\kern 0pt}\nolimits}}
\def\BB{{\mathop{\mathbb B\kern 0pt}\nolimits}}
\def\PP{{\mathop{\mathbb P\kern 0pt}\nolimits}}
\newtheorem{thm}{Theorem}
\newtheorem{defi}{Definition}[section]

\newtheorem{lem}[defi]{Lemma}
\newtheorem{rmk}[defi]{Remark}
 
\newtheorem{prop}[defi]{Proposition}

\def\R {\mathbb{R}}
\def\N {\mathbb{N}}
\def\Z {\mathbb{Z}}

\def \and {\quad\hbox{and}\quad}
\newcommand{\h}{{\rm h}}

\numberwithin{equation}{section}

\begin{document}


  \title[The three-dimensional Navier-Stokes equations without vertical viscosity ]
  {Large, global solutions to the three-dimensional the Navier-Stokes equations without vertical viscosity}

\author[I. Gallagher]{Isabelle Gallagher}
\address[I. Gallagher]%
{DMA, \'Ecole normale sup\'erieure, CNRS, PSL Research University\\
45 rue d'Ulm, 75005 Paris, 
and Universit\'e de Paris, 75006 Paris, France}
\author[A. Yotopoulos]{Alexandre Yotopoulos}
\address[A. Yotopoulos]%
{Universit\'e de Paris and Sorbonne Universit\'e, CNRS, IMJ-PRG,   75013 Paris, France}

 \keywords{Navier-Stokes equations   ; anisotropy;  Besov spaces}
\begin{abstract}
The three-dimensional, homogeneous, incompressible Navier-Stokes equations are studied in the absence of viscosity in one direction. It is shown that there are arbitrarily large initial data generating a unique global solution, the main feature of which is that they are slowly varying in the direction where viscosity is missing. The difficulty arises from the complete absence of a regularising effect in this direction. The special structure of the nonlinear term, joint with the divergence-free condition on the velocity field, is crucial in obtaining the result.
  \end{abstract}

 \maketitle
 

 

\section{Introduction  and énoncé du résultat}   \label{first}

The incompressible homogeneous Navier-Stokes equations in $d$ space dimensions are written as
$$
{\rm(NS)} \quad  \left\{
\begin{array}{l}
\partial_{t} u + u\cdot \nabla u - \Delta u= -\nabla p  \\
\mbox{div} \, u= 0\\
u_{|t=0} = u_{0}\, ,
\end{array}
\right.
$$
where~$p=p(t,x)$ and~$u=(u^1,\dots,u^d)(t,x)$ are respectively the pressure and the velocity field of a viscous incompressible fluid. The viscosity has been set to~1 for simplicity.
 We are interested here in the case where the viscosity of the fluid is strongly anisotropic: the Laplacian acts only in the horizontal coordinates. By defining$$
\Delta_\h := \partial_1^2 + \partial_2^2\, ,
$$
 the system of equations writes
$$
{\rm(NS)}_\h \quad  \left\{
\begin{array}{l}
\partial_{t} u + u\cdot \nabla u -\Delta_\h u= -\nabla p  \quad \mbox{in} \quad
\R^+ \times \Omega\\
\mbox{div} \, u= 0\\
u_{|t=0} = u_{0}\, .
\end{array}
\right.
$$
 The spatial domain~$\Omega$ will be in the following~$\TT^2 \times \R$ or~$\TT^3$, where~$\TT^d:= (\R_{/\Z})^d $ is the~$d$-dimensional torus.  The assumption of zero vertical viscosity originates from the study of geophysical fluids, notably the oceans where the viscosity, known as ``turbulent'', is often much weaker in the vertical variable~\cite{cdgg}.   
 
 \medskip

Before presenting the results obtained in this article concerning~${\rm(NS)}_\h $, let us recall some known results on these equations
starting with the case of~(NS), where the
  viscosity is isotropic (we refer for example to~\cite{bcd,lemarie,lemarie2} for details and more references). In the case where the Laplacian acts in the three directions of space it is well known since the work of J. Leray~\cite{leray} that for any initial data~$u_0$ in the space~$L^2(\Omega)$ there exists a global distributional solution~$u$ to~(NS), of finite energy in the sense that
$$
\forall t \geq 0 \,, \quad \frac12 \|u(t)\|_{L^2(\Omega)}^2 + \int_0^t \|\nabla u(t')\|_{L^2(\Omega)}^2 \, dt'\leq \frac12 \|u_0\|_{L^2(\Omega)}^2, .
$$
 Let us recall that in dimension two space J. Leray shows in~\cite{leray2D} the uniqueness of this finite energy solution (and it verifies the energy equality). Uniqueness in the three dimensional case is an open problem, although recent results (see~\cite{gs} for numerical evidence, \cite{bv} for  distributional solutions and~\cite{abc}   in the presence of forcing) tend to indicate that the Leray solution to (NS) may well be non unique.

In the anisotropic context of~${\rm(NS)_h}$, on the other hand, the absence of compactness in the vertical direction makes the proof of~\cite{leray} inoperative and the global existence of weak solutions is not known. 

\bigskip
 
Concerning the existence of unique solutions in the isotropic framework, the Cauchy  theory is, as often for evolution PDEs, related to the scale invariance of the equation: for all~$\lambda >0$,
  if~$u$ is a solution of~(NS) associated to the data~$u_0$
  then~$\lambda u(\lambda^2 t, \lambda x )$ is a solution of~(NS) associated to the data~$\lambda u_0(\lambda x )$.  This scale invariance remains for~${\rm(NS)_h}$. 
Let us give some examples of scale invariant spaces for the initial data: first we recall the definition of   homogeneous Sobolev spaces~$H^s(\R^d)$, given by the norm (for~$s<d/2$)
$$
\|f\|_{H^s(\R^d)}:=\Big(\int_{\R^d} |\widehat f(\xi)|^2 \, |\xi|^{2s} \, d\xi
\Big)^\frac 12
$$
where~$\widehat f$ is the Fourier transform of~$f$.  In the case of periodic or hybrid boundary conditions  considered in this article, the definition becomes
$$
\|f\|_{H^s(\T^3)}:=\Big(\sum_{n \in \N^3} |\widehat f(n)|^2 \, |n|^{2s} 
\Big)^\frac 12 \quad \mbox{and}\quad\|f\|_{H^s(\T^2\times \R)}:=\Big(\sum_{n \in \N^2} \int_\R |\widehat f(n,\xi)|^2 \, (|n|^2 + |\xi|^2)^{s}  \, d\xi
\Big)^\frac 12\, .
$$
The spaces~$ H^\frac12(\Omega)$ and~$L^3(\Omega)$ are scale invariant. 
  The existence of unique solutions to~(NS), in short time (globally in time under a smallness condition on the initial data), is known for an initial data in the space~$  H^\frac12$ since H. Fujita and T. Kato~\cite{fk}, in~$L^3$ since~\cite{gigamiyakawa,kato,weissler}. In the framework of Besov spaces~$B^{-1+\frac3p}_{p,\infty}$ for~$p<\infty$ (see Definition~\ref{debesoviso} below) we know that a global solution exists for small data since~\cite{Planchon}. The best result 
  in this context is due to H. Koch and D. Tataru~\cite{kochtataru}: these authors prove by a fixed point argument (as it is the case for all the uniqueness results mentioned above)
  that   (NS) is globally wellposed under a smallness condition on
   $$
 \|u_0\|_{\rm{BMO}^{-1} }:= \sup_{t >0} t^\frac12 \|e^{t\Delta} u_0\|_{ L^\infty }  
 + \supetage{x\in \R^3}{R>0} \frac 1 {R^\frac 32}\Big( \int_{[0,R^2] \times B(x,R)} |(e^{t\Delta} u_0)(t,y)|^2 \, dydt\Big)^\frac12 \, ,
  $$    
  where~$B(x,R)$ is the ball centred at~$x$ with radius~$R$.
  The space~$\text{BMO}^{-1}$, like the other spaces mentioned above, 
  is invariant by the change of scale of the equation. Note that the norm~$\ds \sup_{t >0} t^\frac12 \|e^{t\Delta} u_0\|_{ L^\infty }$ which appears above is equivalent to the Besov norm~$  \|u_0\|_{B^{-1 }_{\infty, \infty} }$ (see~(\ref{equivnorm}) for the equivalence).  This space~$B^{-1 }_{\infty, \infty} $ 
  is in fact the space into which any Banach space of scale-invariant tempered distributions is embedded, see~\cite{meyer} --- if we want to define a notion of  ``large'' initial data for (NS), it is thus in~$B^{-1 }_{\infty, \infty} $ that it should be measured; we refer the reader to the Appendix for more information on these spaces, whose definition we recall below -- as well as that of the anisotropic Besov spaces which are used in this article and which are modelled on them.  We note~$\xi = (\xi_\h,\xi_3) = (\xi_1,\xi_2,\xi_3)$ with~$\xi_\h\in \Z^2$ and~$\xi_3 \in \R$ the Fourier variables on~$\T^2 \times \R$ (and~$\xi_3 \in \ZZ$ if~$\Omega = \T^3$). 
  \begin{defi}[Isotropic Besov spaces]\label{debesoviso}
 Let~$\chi$ be a radial function in~$\mathcal{D}(\mathbb{R})$ such that~$\chi(t) = 1$ for~$|t| \leq 1$
and~$\chi(t) = 0$ for~$|t|>2$.  For all~$q \in \ZZ $ we define the frequency truncation operators    
$$
\begin{aligned}
\widehat {S_q f} (\xi)
:= \chi \bigl(2^{-q}|\xi|\bigr) \widehat f(\xi)
\quad \mbox{and} \quad \Delta_q :=
 S_{q+1} - S_{q} \, .
 \end{aligned}$$
For any~$p,r$ in~$ [1,\infty]$ and any~$s$ in~$ \R, $ with~$s < 3/p$ (or~$s \leq 3/p$ if~$r=1$), the homogeneous Besov space~$B^{s}_{p,r}$ is the space of tempered distributions~$f$ such that    $$
    \|f\|_{B^{s} _{p,r}}:= \Big\| 2^{qs}\| \Delta_{q}  f\|_{L^p} \Big\|_{\ell^r}< \infty \, .
    $$
 \end{defi}
The Sobolev space~$  H^s$ corresponds to the choice~$p=r=2$. It is well-known (see~\cite{bcd} for instance) that the Besov norm has an equivalent formulation via the heat flow
\beq\label{equivnorm}
\forall s<0 \, , \quad f \in     \|f\|_{B^{s} _{p,r} } \Longleftrightarrow  \Big\|t^{-\frac s2} \|e^{t\Delta}f\|_{L^p} \Big\|_{L^r (\frac {dt}t)}< \infty \, .
\eeq
\smallskip

The anisotropic Besov spaces used in this text are of two types, defined below.
   \begin{defi}[Anisotropic Besov spaces]\label{deflpanisointro}
 With the notations of Definition~{\rm\ref{debesovis}}, the horizontal frequency truncation operators are defined for $j\in \ZZ $ by
$$
\begin{aligned}
\widehat {S_j^{\rm h} f}  (\xi)
:= \chi \bigl(2^{-j}|\xi_\h|\bigr) \widehat f(\xi)
  \and  \Delta_j^{\rm h} :=
 S_{j+1}^{\rm h} - S_{j}^{\rm h} \, ,
 \end{aligned}$$
and the vertical frequency truncations for $q\in \ZZ$ by
$$
\begin{aligned}
\widehat {S_q^{\rm v} f}  (\xi)
  := \chi (2^{-q}|\xi_3|) \widehat f(\xi)
 \and  \Delta_q^{\rm v} :=
 S_{q +1}^{\rm v} - S_{q}^{\rm v}
  \, .
 \end{aligned}$$
 For all~$s \leq 2/p$ and~$ s' \leq 1/p$, the Besov     space ~${\mathcal  B}^{s,s'}$ is the space of tempered distributions~$f$ such that
    $$
    \|f\|_{{\mathcal B}^{s,s'} }:=\sum_{j,q} 2^{js + qs'}\|\Delta_j^{\rm h} \Delta_{q}^{\rm v} f\|_{L^2} < \infty \, ,
    $$
    and the Besov space~${ B}^{0,s } $ is the space of tempered distributions~$f$ such that
    $$
    \|f\|_{{ B}^{0,s} }:= \sum_{ q} 2^{ qs }\| \Delta_{q}^{\rm v} f\|_{L^2} < \infty \, .
    $$
      \end{defi}
Let us notice that ${\mathcal B}^{0,s }$ is continuously embedded in ~${ B}^{0,s } $, since
  \beq\label{injcont} 
  \sum_{ q}   2^{ qs } \|  \Delta_{q}^{\rm v}  f\|_{L^2} \leq   \sum_{ j,q}   2^{ qs } \|  \Delta_j^{\rm h}  \Delta_{q}^{\rm v}  f\|_{L^2}\, .
  \eeq
%
%
 Note that the anisotropic Besov space~$B^{0,\frac12}$ appears naturally here because it is modelled on the space~$L^2(\T^2; \dot H^\frac12(\R))$ -- which is a natural space  in the context of~${\rm(NS)_h}$ since~$L^2(\T^2)$ is associated with the two-dimensional equation and~$\dot H^\frac12$ is scale-invariant in dimension three -- while being a Banach space (unlike~$\dot H^\frac12(\R)$ for example, and thus~$L^2(\T^2; \dot H^\frac12(\R))$).    
This space was introduced in this context by M. Paicu in~\cite{paicu}, who showed the global existence in time of solutions to~${\rm(NS)_h}$ for small data (local in time for any data) in~$B^{0,\frac12}$. He also showed the uniqueness of solutions in~$L^\infty(\R^+,B^{0,\frac12})$ whose horizontal gradient is in~$ L^2(\R^+,B^{0,\frac12})$. 
 
\medskip

In the present work we seek, in the spirit of the works~\cite{cg3,cgm,cgz} for example, to exhibit  initial data which may be arbitrarily large but for which
there is a unique associate global solution to~${\rm(NS)_h}$.
 The natural context, following these works, is to consider initial data varying slowly in one direction, and the specificity of this paper is to assume that this direction is the same as the one in which there is no viscosity (the vertical one for instance).
 This work thus follows a series of works concerning either the slowly varying case or the anisotropic equation (see for example~\cite{bcg,bg,cdgg2,cg3,cgm,cgz,cz,cz2,gz,Iftimie,paicu,paicuzhang}).  To our knowledge, this is the first time that the slowly variable character in one direction, which allows to obtain global solutions without any smallness assumption on the initial data, is mixed with the absence of vertical viscosity in the equation. This leads to be particularly careful in the estimates since no regularising effect is possible in the vertical direction. 
 In particular the special structure of the nonlinear term, joint with the condition that the velocity field is divergence free, will be crucial in the analysis.
 The result is as follows.

 \begin{thm}\label{thmyotopoulos}
  Let~$ u^\h_{0} = ( u^1_{0} , u^2_{0} ) $ and~$w_0= ( w^1_{0} ,w^2_{0} , w^3_{0} ) $ be two divergence free vector fields  with~$( u^\h_{0} , w^3_{0} )$ belonging to~${\mathcal B}^{0,\frac12} \cap{\mathcal B}^{-1,\frac52}$.
Let, for all~$\e\in (0,1)$,
 \begin{equation}\label{initialdata}
u^\e_{0}(x) := (u_0^\h +\varepsilon w_0^\h,w_0^3) (x_1,x_2,\varepsilon x_3) \, .
\end{equation}
For~$\e$ small enough, there is a unique global solution~$u^\e$ to~${\rm(NS)_h}$ associated to the initial data~{\rm(\ref{initialdata})}, in the space~$L^\infty(\R^+,B^{0,\frac12})$ and such that~$\nabla^\h u^\e$ belongs to~$ L^2(\R^+,B^{0,\frac12})$.
 \end{thm}
 
  \begin{rmk}
It is shown in~\cite{cg3} that a function of the form
 $$
 h^\e(x) = f(x_\h) g(\e x_3)
 $$
 with~$f$ and~$g$ in the Schwartz class, verifies
 $$
\|h^\e\|_{ B^{-1}_{\infty,\infty} } \geq \frac 1 4 \|f\|_{\dot B^{-1}_{\infty,\infty} } 
\|g\|_{L^\infty} 
$$
so can be as large as desired in~$ B^{-1}_{\infty,\infty}$.
  \end{rmk}

\begin{rmk}
The periodic character of the horizontal variable allows us to obtain good regularity estimates    on~$w^\h$ from estimates on~$w^3$ via the identity
$$
w^\h = - \nabla^\h (\Delta_\h)^{-1} \partial_3 w^3 \, , \quad \nabla^\h := (\partial_1,\partial_2) \, ,
$$
because~$\partial_3 w^3$ has a zero horizontal mean (see Section~{\rm\ref{demforceext}}). 
In the case where~$w_0 \equiv 0$, the proof of Theorem~{\rm\ref{thmyotopoulos}} shows that we can assume indifferently that the horizontal variable is in~$\R^2$ or in~$\T^2$.
    \end{rmk} 
   From now on we note
$$
 \big[f \big]_\e (x) := f(x_{\rm h},\e x_3) \, .
$$
The method of proving the Theorem~\ref{thmyotopoulos}
consists in looking for the solution~$u^\e$, and the associated pressure~$ p^\e$, in the form
 \begin{equation}\label{defapp}
u^\e = u^\e_{\rm app}+R^\eps \, , \quad p^\e = p^\e_{\rm app}+q^\eps
\end{equation}
with
$$
u^\e_{\rm app} =\big[u^\h +\varepsilon w^\h,w^3\big]_\e \, , \quad p^\e_{\rm app} =\big[ p_0+ \e p_1\big]_\e$$  
where for all~$y_3$,~$u^\h (\cdot, y_3) $ 
is a solution of the two-dimensional Navier-Stokes equations with initial data~$u_0^\h (\cdot, y_3)$:
$$
{\rm(NS2D)}_{y_3} \quad \left\{
\begin{array}{l}
\partial_{t} u^\h+ u^\h \cdot \nabla^\h u^\h -\Delta_\h u^\h= -\nabla^\h p_0\quad \mbox{in} \quad
\R^+ \times \T^2\\
\mbox{div}_\h u^\h = 0\\\
u^\h _{|t=0} = u^\h_{0} (\cdot,y_3)\, ,
\end{array}
\right.
$$
and~$ w$ is a solution of the linear equation
$$
{\rm(T)}  \quad \left\{
\begin{array}{l}
\partial_{t} w + u^\h\cdot \nabla^\h w -\Delta_\h w = -( \nabla^\h p_1 , 0)\quad \mbox{in} \quad
\R^+ \times \Omega \\
\mbox{div}  \, w= 0\\
w_{|t=0} = w_{0} \, .
\end{array}
\right.
$$

\medskip

Since~$ u^\h_{0}  (\cdot,y_3)$ belongs to~$ L^2(\T^2)$, then given~$y_3$ there is a unique global solution~$u^\h$  to ${\rm(NS2D)}_{y_3}$,  in the energy space~$L^\infty(\R^+;L^2(\T^2) )\cap L^2(\R^+;\dot H^1(\T^2))$.
The vector field~$w$ also exists uniquely for all times (this will be clarified in Section~\ref{demforceext}), 
and the main part of the work consists therefore in solving globally in time, for~$\e$ sufficiently small, the (perturbed anisotropic Navier-Stokes) equation verified by~$R^\e$.

\medskip

The plan of the article is as follows. In Section~\ref{find the proof} we show that for~$\e$ sufficiently small, the equation verified by~$R^\e$ has a global solution: this proof relies on a priori estimates on the approximate solution~$ u^\e_{\rm app}$ which are derived in Section~\ref{demforceext} from estimates on~${\rm(NS2D)}_{y_3} $
 and on~${\rm(T)}$. The Appendix is devoted to the recollection of classical results concerning the functional spaces appearing in the present work, as well as two important trilinear estimates which can be found in the literature.

\medskip

Throughout this article and unless otherwise stated, we will note by~$C$ a universal constant, in particular independent of~$\e$, which can change from one line to another. We will sometimes note~$a \lesssim b$ if~$a \leq Cb$. We will note~$q\sim q'$ if~$q \in [q'-C,q'+C]$.  
We will note generically by~$(s_q)_{q\in \Z}$ a sequence of positive real numbers such that~$\displaystyle \sum_{q \in \ZZ} s_q^\frac12 \leq 1$ and by~$(d_q)_{q\in \Z}$ a sequence of positive real numbers such that~$\displaystyle \sum_{q \in \ZZ} d_q \leq 1$.
Finally, if~$X$ is a function space on~$\T^2$ and~$Y$ a function space on~$\R$, we write $X_\h Y_\v:=X(\T^2;Y(\R))$   and similarly~$ Y_\v X_\h:=Y(\R;X(\T^2))$.

 \section{Proof of the theorem}{\label{find the proof}}
  \subsection{Main steps of the proof}
 Recalling the notation~(\ref{defapp}), let us write the system of equations verified by~$R^\e$. We have
  \beq
 \label{eqReps}
 \partial_{t} R^\e + R^\e \cdot \nabla R^\e+u^\e_{\rm app} \cdot \nabla R^\e+R^\e\cdot \nabla u^\e_{\rm app} -\Delta_\h R^\e= - \nabla q_\e +\e F^\e
 \eeq
 with
 $$ F^\e :=  \e {\color{black}\Bigl[}\Bigl(w^{\rm h}  \cdot \nabla^{\rm h} (w^{\rm h} ,0 )+  w^3 \partial_3( w^{\rm h},0)\Bigr){\color{black}\Bigr]_\e} +   {\color{black}\Bigl[}\Bigl(w \cdot\nabla  (u^\h,w^3) \Bigr){\color{black}\Bigr]_\e}  +   \big (0,[\partial_3 (p_0 +  \e p_1)]_\e\big)    \, ,
$$
and $ R^\e_{|t=0} = 0$.

\medskip

To prove the global existence of~$R^\e$ we   write an a priori estimate on~$ \| R^\e\| _{\tilde L^\infty(\R^+;B^{0,\frac12})}$ and~$ \|\nabla^\h R^\e\|_{\tilde L^2(\R^+; B^{0,\frac12})}$ (the definition of these spaces is recalled in the Appendix) and omit  the classical step of regularization of the system to justify the estimates. Moreover, as recalled in the introduction, only the global existence of solutions has to be proved since the uniqueness for~${\rm (NS)_h}$ in our functional framework is a consequence of~\cite{paicu}.

\medskip

In order to absorb the linear terms in~(\ref{eqReps}) we   use a Gronwall-type argument, but in the context of~$\tilde L^p$ spaces in time (see the Appendix for a definition of these spaces). One strategy (see~\cite{gip} for example) is to write a partition of~$\R^+$ into time intervals
\begin{equation}\label{partitiontemps}\R^+= \bigcup_{k=0}^{K-1} [t_k,t_{k+1}[
\end{equation}
as in   Proposition~\ref{propNS2D} below, which depends on a constant~$\bar C$ which will be fixed at the end. 
 We then write, following~\cite{paicu}, an energy estimate in~$L^2$ on~$\Delta_q^\v R^\e$ and it comes after integration on a time interval~$ {[t_k,t_{k+1}]}$ 
\begin{equation}\label{energyestimate}
\begin{aligned}
& \frac12  \| \Delta_q^\v R^\e(t_{k+1})\|_{L^2}^2  + \int_{t_k}^{t_{k+1}}    \|\nabla^\h \Delta_q^\v R^\e (t)\|_{L^2}^2\, dt\leq\frac12  \| \Delta_q^\v R^\e(t_k)\|_{L^2}^2   \\
& +   \int_{t_k}^{t_{k+1}}   \big | ( \Delta_q^\v (R^\e\cdot \nabla R^\e)  |
 \Delta_q^\v R^\e)_{L^2}(t)
\big | 
 \, dt \\
& + \int_{t_k}^{t_{k+1}} \Big( \big | ( \Delta_q^\v (u^\e_{\rm app}\cdot \nabla R^\e)  |
 \Delta_q^\v R^\e)_{L^2}(t)
 \big |+  \big | ( \Delta_q^\v ( R^\e\cdot \nabla u^\e_{\rm app}) |
 \Delta_q^\v R^\e)_{L^2}(t)
\big |  \Big) \, dt'\\
& +  \quad  \e \int_{t_k}^{t_{k+1}}\big |  (\Delta_q^\v F^\e  |
 \Delta_q^\v R^\e)_{L^2}
(t)
\big |  \, dt\, .
\end{aligned}
\end{equation}
We note that~$ R^\e(t_0) = R^\e( 0) = 0$. Let us introduce the notation
$$
\tilde L^r_{k} X:= \tilde L^r([t_k,t_{k+1}]; X) \, .
$$
From~(\ref{estimRRappendix}) we know that
\begin{equation}\label{estimRR}
 \int_{t_k}^{t_{k+1}} \big | ( \Delta_q^\v (R^\e\cdot \nabla R^\e ) |
 \Delta_q^\v R^\e)_{L^2}(t)
\big | 
 dt\lesssim 2^{-q}s_q \|\nabla^\h R^\e\|_{\tilde L^2_{k} B^{0,\frac12}}^2 \| R^\e\|_{\tilde L^\infty_{k} B^{0,\frac12}} 
\end{equation}
and from~(\ref{estimRuappendix}) 
$$
\begin{aligned}
 & \int_{t_k}^{t_{k+1}}  \big | ( \Delta_q^\v (u^\e_{\rm app}\cdot \nabla R^\e ) |
 \Delta_q^\v R^\e)_{L^2}(t)
\big | 
 dt\lesssim   2^{-q}s_q \| R^\e\|^\frac12 _{\tilde L^\infty_{k}  B^{0,\frac12}}\|\nabla^\h R^\e\|_{\tilde L^2_{k}  B^{0,\frac12}} \\
 &\quad \times \Big(\|\nabla^\h R^\e\|^\frac12_{\tilde L^2_{k}  B^{0,\frac12}}
  \| u^\e_{\rm app}\|^\frac12 _{\tilde L^\infty_{k}  B^{0,\frac12}}\|\nabla^\h u^\e_{\rm app}\|^\frac12_{\tilde L^2_{k}  B^{0,\frac12}} 
   +\|\nabla^\h u^\e_{\rm app}\| _{\tilde L^2_{k}  B^{0,\frac12}} 
 \| R^\e\|^\frac12 _{\tilde L^\infty_{k}  B^{0,\frac12}}\Big)
\end{aligned}
$$
by recalling that~$(s_q)_{q \in \ZZ}$ denotes generically a sequence of positive real numbers verifying $$
\sum_{q \in \ZZ} s_q^\frac12 \leq 1 \, .
$$
Thanks to Young's inequality
 \begin{equation}\label{young}
ab \lesssim a^p + b^{p'} \, , \quad \frac1p+\frac1{p'} =1
 \end{equation}
this last inequality can also be written
\begin{equation}\label{estimRuapp}
\begin{aligned}
& \int_{t_k}^{t_{k+1}}   \big | ( \Delta_q^\v (u^\e_{\rm app}\cdot \nabla R^\e ) |
 \Delta_q^\v R^\e)_{L^2}
(t)\big | 
 dt\leq    2^{-q}s_q\Big( 
 \frac 1{4}\|\nabla^\h R^\e\|_{\tilde L^2_{k}  B^{0,\frac12}} ^2  \\
& \qquad\qquad\qquad \qquad+ C
  \| R^\e\|  _{\tilde L^\infty_{k}  B^{0,\frac12}} ^2 \|\nabla^\h u^\e_{\rm app}\| _{\tilde L^2_{k}  B^{0,\frac12}} ^2 (1+\|  u^\e_{\rm app}\|  _{\tilde L^\infty_{k}  B^{0,\frac12}}^2 )\Big)\, .
\end{aligned}
\end{equation}
The end of the proof of the theorem relies on the following two lemmas, which are proved respectively in Paragraph~\ref{demlemproduit}   and in Section~\ref{demforceext}.
\begin{lem}\label{demlemproduct}
There exists a constant~$C>0$ such that
under the hypotheses of   Theorem~{\rm{\ref{thmyotopoulos}}}, there exists a sequence~$(s_q)_{q \in \ZZ}$ of positive real numbers verifying~$\displaystyle \sum_{q \in \ZZ} s_q^\frac12 \leq 1$ and such that
\begin{equation}\label{estimuappR}
\begin{aligned}
 & \int_{t_k}^{t_{k+1}}   \big | ( \Delta_q^\v (R^\e\cdot \nabla u^\e_{\rm app} )  |
 \Delta_q^\v R^\e)_{L^2}(t)
\big | 
 dt \leq   2^{-q}s_q\Big( 
 \frac 1{4}\|\nabla^\h R^\e\|_{\tilde L^2_{k}  B^{0,\frac12}} ^2  \\
&\quad  + C
  \| R^\e\|  _{\tilde L^\infty_{k}  B^{0,\frac12}} ^2 \big (\| u^\e_{\rm app}\| _{\tilde L^2_{k}   {\mathcal B}^{1,\frac12}} ^2 (1+\|  u^\e_{\rm app}\|  _{\tilde L^\infty_{k}  B^{0,\frac12}}^2 )+  \| \partial_3 u^\e_{\rm app} \|_{\tilde L^1_{k}  {\mathcal B}^{1,\frac12}}\big ) \Big)
 \, .
\end{aligned}
\end{equation}
\end{lem}
\begin{lem}\label{lemdemforceext}
There exists a constant~$C>0$ such that
under the hypotheses of    Theorem~{\rm{\ref{thmyotopoulos}}}, there exists a sequence~$(s_q)_{q \in \ZZ}$ of positive real numbers  verifying~$\displaystyle \sum_{q \in \ZZ} s_q^\frac12 \leq 1$ and such that
\begin{equation}\label{estimF}
2^q \int_{t_k}^{t_{k+1}}\big | (\Delta_q^\v F^\e |
 \Delta_q^\v R^\e)_{L^2} (t)\big | \, dt\leq C s_q \| R^\e\| _{\tilde L^\infty_{k} B^{0,\frac12}}  
\, . \end{equation}
\end{lem}
Let us return to~(\ref{energyestimate}). By gathering~(\ref{estimRR}),~(\ref{estimRuapp}),~(\ref{estimuappR}) and~(\ref{estimF}) we get
$$
\begin{aligned}
& \frac{2^q}2   \| \Delta_q^\v R^\e(t_{k+1})\|_{L^2}^2  +2^q \int_{t_k}^{t_{k+1}}    \|\nabla^\h \Delta_q^\v R^\e (t)\|_{L^2}^2\, dt\leq  \frac{2^q}2   \| \Delta_q^\v R^\e(t_{k })\|_{L^2}^2   \\
 & +Cs_q \|\nabla^\h R^\e\|_{\tilde L^2_k B^{0,\frac12}}^2  \|  R^\e\|_{\tilde L^\infty_{k}   B^{0,\frac12}}   + C\e s_q
  \| R^\e\|  _{\tilde L^\infty_{k}  B^{0,\frac12}} + \frac{s_q}{2}  \|\nabla^\h R^\e\|_{\tilde L^2_k B^{0,\frac12}}^2\\
  &+ Cs_q
  \| R^\e\|  _{\tilde L^\infty_{k}  B^{0,\frac12}} ^2 \big (\| u^\e_{\rm app}\| _{\tilde L^2_{k}  {\mathcal B}^{1,\frac12}} ^2 (1+\|  u^\e_{\rm app}\|  _{\tilde L^\infty_{k}  B^{0,\frac12}}^2 )+  \|  \partial_3 u^\e_{\rm app} \|_{\tilde L^1_{k}  {\mathcal B}^{1,\frac12}}\big )\, .
   \end{aligned}
$$
We used the fact, recalled in~(\ref{estimate11}), that
$$
\|\nabla^\h a\|_{B^{0,\frac12}} \lesssim \|  a\|_{{\mathcal B}^{1,\frac12} } \, .
$$
By taking the square root of the two sides of the equation and summing over~$q\in \Z$ we find that 
$$
 \begin{aligned}
 \| R^\e\|  _{\tilde L^\infty_{k}  B^{0,\frac12}} &+ \|\nabla^\h R^\e\|_{\tilde L^2_k B^{0,\frac12}}  \leq  \| R^\e(t_{k })\|  _{ B^{0,\frac12}}\\
& \quad+  \| R^\e\|  ^\frac12_{\tilde L^\infty_{k}  B^{0,\frac12}}\big( \|\nabla^\h R^\e\|_{\tilde L^2_k B^{0,\frac12}}+ C\sqrt \e\big) \\
&\quad + C  \| R^\e\|  _{\tilde L^\infty_{k}  B^{0,\frac12}} \Big (\| u^\e_{\rm app}\| _{\tilde L^2_{k}  {\mathcal B}^{1,\frac12}}   (1+\|  u^\e_{\rm app}\|  _{\tilde L^\infty_{k}  B^{0,\frac12}}  )+  \|  \partial_3 u^\e_{\rm app} \|_{\tilde L^1_{k}  {\mathcal B}^{1,\frac12}}^\frac12\Big )\, .
   \end{aligned}
$$It is then sufficient to choose the partition~(\ref{partitiontemps}) thanks to   Proposition~\ref{propNS2D} so that
$$
\| u^\e_{\rm app}\| _{\tilde L^2_{k} {\mathcal B}^{1,\frac12}}   (1+\| u^\e_{\rm app}| _{\tilde L^\infty_{k} B^{0,\frac12}} ) \leq \frac1{4 C} =: \frac1{\bar C}
$$
and to choose, still thanks to   Proposition~\ref{propNS2D}, $\e$ small enough so that
$$
  \|  \partial_3 u^\e_{\rm app} \|_{\tilde L^1_{k}  {\mathcal B}^{1,\frac12}}^\frac12\leq\frac1{4 C}\, \cdotp
 $$
 Then the above inequality becomes
 $$
 \frac12 \| R^\e\|  _{\tilde L^\infty_{k}  B^{0,\frac12}} + \|\nabla^\h R^\e\|_{\tilde L^2_k B^{0,\frac12}}  \leq C\| R^\e(t_{k})\|  _{ B^{0,\frac12}}
 + C \| R^\e\|  ^\frac12_{\tilde L^\infty_{k}  B^{0,\frac12}} \| \nabla^\h R^\e\|_{\tilde L^2_k B^{0,\frac12}}+ C  \e \, .
$$
Let now~$T^\e$ be the maximum time for which
$$
 \forall t \leq T^\e \, , \quad \| R^\e\|  ^\frac12_{\tilde L^\infty ([0,t];B^{0,\frac12})} \leq  \frac {1}{2C} \, \cdotp
$$
Then 
as long as~$t_{k+1} \leq T^\e$ we have
$$
  \frac12 \| R^\e\| _{\tilde L^\infty_{k} B^{0,\frac12}} +\frac12 \|\nabla^\h R^\e\|_{\tilde L^2_k B^{0,\frac12}} \leq C \| R^\e(t_{k })\| _{ B^{0,\frac12}} + C \e
$$
and as~$ \| R^\e(t_{0})\| _{ B^{0,\frac12}} = 0$, by iterating~$K$ times this inequality we find (see~\cite{gip,bg} for example) that~$T^\e= \infty$ and that there exists a constant~$C_0$ (depending on~$u_0^\h$ and~$w_0$ via   Proposition~\ref{propNS2D}) such that
$$
  \| R^\e\| _{\tilde L^\infty(\R^+;B^{0,\frac12})} +\frac12 \|\nabla^\h R^\e\|_{\tilde L^2(\R^+;B^{0,\frac12})} \lesssim \e \exp C_0 \, .
$$
This concludes the proof of   Theorem~{\rm{\ref{thmyotopoulos}}}.  \qed

\subsection{Proof of Lemma~\ref{demlemproduct}} \label{demlemproduit}
We begin by noting that since the divergence of $R^\e$ is zero
 $$
 - \big( \Delta_q^\v (R^\e\cdot \nabla u^\e_{\rm app} )  |
 \Delta_q^\v R^\e\big)_{L^2} = \sum_{\ell = 1}^2  \big( \Delta_q^\v (R^{\e,\ell} u^\e_{\rm app} )  |\partial_\ell
 \Delta_q^\v R^\e\big)_{L^2} +   \big( \Delta_q^\v (R^{\e,3} u^\e_{\rm app} )  |\partial_3
 \Delta_q^\v R^\e\big)_{L^2}  \, .
  $$
We set
  $$
  I_q:=   \big |\sum_{\ell = 1}^2 \big ( \Delta_q^\v (R^{\e,\ell} u^\e_{\rm app} )  |\partial_\ell
 \Delta_q^\v R^\e\big)_{L^2} \big |  \and  J_q:=  \big |\big ( \Delta_q^\v (R^{\e,3} u^\e_{\rm app} )  |\partial_3
 \Delta_q^\v R^\e\big)_{L^2}  \big |\, .
  $$
 Let us start by studying the contribution of~$I_q$. There holds
  \begin{equation}\label{estimateIq}
 \begin{aligned}
  2^{ q}\int_{t_k}^{t_{k+1}}  I_q(t) \, dt& \lesssim    2^{ \frac q2}\big\| \Delta_q^\v (R^{\e,\h}  u^\e_{\rm app} )\big\|_{L^2_{k}  L^2}  2^{ \frac q2}\big\| \Delta_q^\v \nabla^\h R^\e \big\|_{L^2_{k}  L^2}\\
 & \lesssim    2^{ \frac q2}\big\| \Delta_q^\v (R^{\e,\h} u^\e_{\rm app} )\big\|_{L^2_{k}  L^2}  d_q  \| \nabla^\h R^\e\| _{\tilde L^2_{k}B^{0,\frac12}}
\end{aligned}
\end{equation}
with the generic notation presented in the introduction: $(d_q)_{q\in \Z}$ is a sequence of positive real numbers such that $$
\sum_{q \in \ZZ} d_q  \leq 1 \, .
$$

We also recall the notation~$\tilde L^r_{k}X =\tilde L^r([t_{k },t_{k+1}];X)$.
We then use the Bony decomposition into paraproduct and remainder~(\ref{paraproduit}) which allows us to write
\begin{equation}\label{firstdecompo}
 \begin{aligned}
2^{ \frac q2}\big\| \Delta_q^\v (R^{\e,\ell} u^\e_{\rm app} )\big\|_{L^2_{k}  L^2} 
& \lesssim 2^{ \frac q2} \sum_{q'\sim q } \|S_{q'-1}^\v R^\e\|_{L^4_{k}  L^4_\h L^\infty_\v } \|\Delta_{q' }^\v   u^\e_{\rm app}\|_{L^4_{k}L^4_\h  L^2_\v } \\
& + 2^{ \frac q2} \sum_{q'\sim q }\|S_{q'-1}^\v  u^\e_{\rm app} \|_{L^4_{k}  L^4_\h L^\infty_\v  } \|\Delta_{q' }^\v R^\e\|_{L^4_{k}   L^4_\h L^2_\v} \\
& + 2^{ q} \sumetage{2^{q'} \gtrsim 2^q} { q'' \sim   {q'}} 
 \|\Delta_{q'' }^\v   u^\e_{\rm app}\|_{L^4_{k}  L^4_\h L^2_\v}  \|\Delta_{q' }^\v R^\e\|_{L^4_{k}  L^4_\h L^2_\v} \\
&=:T_{q\h}^1 +  T_{q\h}^2+ R_{q\h}\, , 
 \end{aligned}
\end{equation}
where we used Bernstein's inequality~(\ref{bernstein}) in the last inequality:
$$
\big\| \Delta_q^\v (R^{\e,\ell} u^\e_{\rm app} )\big\|_{L^2_{k}  L^2} \lesssim  2^{ \frac q2} \big\| \Delta_q^\v (R^{\e,\ell} u^\e_{\rm app} )\big\|_{L^2_{k}  L^2_\h L^1_\v} \, .
$$
Let us estimate each of the terms in succession. For~$T_{q\h}^1$ we start by noting that for any function~$a$ and any~$x_3 \in \R$ we have thanks to the Sobolev embedding~$H^\frac12(\T^2) \subset L^4(\T^2)$, 
 $$\begin{aligned}
\|a(\cdot,x_3)\|_{ L^4(\T^2)} &\lesssim \|a(\cdot,x_3)\|_{H^\frac12(\T^2)} \\
&\lesssim \|a(\cdot,x_3)\|_{L^2(\T^2)}^\frac12 \|  \nabla^\h  a(\cdot,x_3)\|_{L^2(\T^2)}^\frac12
\end{aligned}
$$
so by the Cauchy-Schwarz inequality in~$x_3$ there holds
$$
\|a\|_{L^2_\v L^4_\h} \lesssim \|a\|_{L^2}^\frac12 \|  \nabla^\h  a\|_{L^2}^\frac12 \, .
$$
Using again Bernstein's inequality~(\ref{bernstein}) 
$$
 \|\Delta_{q''}^\v R^\e\|_{L^4_{k}  L^4_\h L^\infty_\v } \lesssim 2^{\frac{q''}2} \|\Delta_{q''}^\v R^\e\|_{L^4_{k}   L^4_\h L^2_\v} 
$$
and then Minkowski's inequality
$$
  \|\Delta_{q''}^\v R^\e\|_{L^4_{k}   L^4_\h L^2_\v} 
\leq   \|\Delta_{q''}^\v R^\e\|_{L^4_{k}    L^2_\v L^4_\h} 
$$
we gather
$$
 \begin{aligned}
 \|S_{q'-1}^\v R^\e\|_{L^4_{k}   L^4_\h L^\infty_\v} &\lesssim \sum_{2^{q''} \lesssim 2^ {q '} } 2^{\frac{q''}2} \|\Delta_{q''}^\v R^\e\|_{L^4_{k}  L^2_\v L^4_\h} \\
  &\lesssim \sum_{2^{q''} \lesssim 2^ {q' } } 2^{\frac{q''}2} \|\Delta_{q''}^\v R^\e\|^\frac12_{L^\infty_{k}  L^2}  \|\Delta_{q''}^\v    \nabla^\h R^\e\|^\frac12_{L^2_{k}  L^2} 
  \\
  &\lesssim  \| R^\e\|^\frac12_{\tilde L^\infty_{k}  B^{0,\frac12}}  \|  \nabla^\h R^\e\|^\frac12_{\tilde L^2_{k}   B^{0,\frac12}}   
 \end{aligned}
$$
 by the Cauchy-Schwarz inequality.  Similarly
 $$
 \begin{aligned}
 \|\Delta_{q' }^\v   u^\e_{\rm app}\|_{L^4_{k}   L^4_\h L^\infty_\v}    &\lesssim 2^{\frac{q'}2} \|\Delta_{q' }^\v   u^\e_{\rm app}\|^\frac12_{L^\infty_{k}  L^2}  \|\Delta_{q' }^\v    \nabla^\h   u^\e_{\rm app}\|^\frac12_{L^2_{k}  L^2} 
  \\
  &\lesssim d_{q' }  \| u^\e_{\rm app}\|^\frac12_{\tilde L^\infty_{k}  B^{0,\frac12}}  \|  \nabla^\h u^\e_{\rm app}\|^\frac12_{\tilde L^2_{k}   B^{0,\frac12}}   
 \end{aligned}
$$ hence finally
 $$
 \begin{aligned}
T_{q\h}^1   \lesssim  d_q \| R^\e\| _{\tilde L^\infty_{k}B^{0,\frac12}}^\frac12\|    \nabla^\h R^\e\| _{\tilde L^2_{k}B^{0,\frac12}}^\frac12   \|  u^\e_{\rm app}\| _{\tilde L^\infty_{k}B^{0,\frac12}}^\frac12\|    \nabla^\h  u^\e_{\rm app}\| _{\tilde L^2_{k}B^{0,\frac12}}^\frac12  \, ,
 \end{aligned}
$$
and symmetrically $$
 \begin{aligned}
T_{q\h}^2 \lesssim d_q \| R^\e\| _{\tilde L^\infty_{k}B^{0,\frac12}}^\frac12\|    \nabla^\h R^\e\| _{\tilde L^2_{k}B^{0,\frac12}}^\frac12  \|  u^\e_{\rm app}\| _{\tilde L^\infty_{k}B^{0,\frac12}}^\frac12\|    \nabla^\h  u^\e_{\rm app}\| _{\tilde L^2_{k}B^{0,\frac12}}^\frac12\, .
   \end{aligned}
$$
Finally~$R_{q\h} $ can be estimated by an analogous argument. One has indeed
  $$
 \begin{aligned}
R_{q\h}     &=  2^{ q} \sumetage{2^{q'} \gtrsim 2^q} { q'' \sim   {q'}} 
 \|\Delta_{q'' }^\v   u^\e_{\rm app}\|_{L^4_{k}  L^4_\h L^2_\v}  \|\Delta_{q' }^\v R^\e\|_{L^4_{k}  L^4_\h L^2_\v} \\
  &\lesssim  \sumetage{2^{q'} \gtrsim 2^q} { q'' \sim   {q'}} 
    2^{  {q-q''} }     2^{ \frac{q''}2}   \|\Delta_{q'' }^\v     u^\e_{\rm app}\|^\frac12_{L^\infty_{k}  L^2}  \|\Delta_{q'' }^\v  \nabla^\h   u^\e_{\rm app}\|^\frac12_{L^\infty_{k}  L^2}  
     2^{ \frac{q''}2}   \|\Delta_{q' }^\v R^\e\|^\frac12_{L^\infty_{k}  L^2}  \|\Delta_{q'}^\v  \nabla^\h R^\e\|^\frac12_{L^\infty_{k}  L^2}  
 \end{aligned}
$$ 
hence by Young's inequality for series $$
 \begin{aligned}
R_{q\h} \lesssim  d_q \| R^\e\| _{\tilde L^\infty_{k}B^{0,\frac12}}^\frac12\|    \nabla^\h R^\e\| _{\tilde L^2_{k}B^{0,\frac12}}^\frac12  \|  u^\e_{\rm app}\| _{\tilde L^\infty_{k}B^{0,\frac12}}^\frac12\|    \nabla^\h  u^\e_{\rm app}\| _{\tilde L^2_{k}B^{0,\frac12}}^\frac12  \, .
 \end{aligned}
$$
In conclusion, by returning to~(\ref{estimateIq}) we obtain
$$
 2^{ q}\int_{t_k}^{t_{k+1}}  I_q(t) \, dt\lesssim s_q  \| R^\e\| _{\tilde L^\infty_{k}B^{0,\frac12}}^\frac12 \| \nabla^\h R^\e\| _{\tilde L^2_{k}B^{0,\frac12}}^\frac32
  \|  u^\e_{\rm app}\| _{\tilde L^\infty_{k}B^{0,\frac12}}^\frac12\|    \nabla^\h  u^\e_{\rm app}\| _{\tilde L^2_{k}B^{0,\frac12}}^\frac12  
$$
and by Young's inequality~(\ref{young})  and~(\ref{estimate11}) we get
\beq\label{estimate1}
 \begin{aligned}
2^{ q}\int_{t_k}^{t_{k+1}}  I_q(t) \, dt&\leq   \frac {s_q}{100}\|\nabla^\h R^\e\|_{\tilde L^2_{k}  B^{0,\frac12}} ^2  \\
 &\quad  + Cs_q
  \| R^\e\|  _{\tilde L^\infty_{k}  B^{0,\frac12}} ^2 \| u^\e_{\rm app}\| _{\tilde L^2_{k}   {\mathcal B}^{1,\frac12}} ^2 \|  u^\e_{\rm app}\|  _{\tilde L^\infty_{k}  B^{0,\frac12}}^2 \, .
 \end{aligned}
\eeq
Let us now study the contribution of~$J_q$. We notice that by~(\ref{bernstein}) 
$$
\|\partial_3\Delta_q^\v  R^\e\|_{L^2} \lesssim  2^{ q}\|\Delta_q^\v  R^\e\|_{L^2} 
$$
hence
 \begin{equation}\label{estimateJq}
 \begin{aligned}
  2^{ q}\int_{t_k}^{t_{k+1}} J_q(t) dt& \lesssim    2^{2q}\big\| \Delta_q^\v (R^{\e,3}  u^\e_{\rm app} )\big\|_{L^1_{k}  L^2}   \big\| \Delta_q^\v  R^\e \big\|_{L^\infty_{k}  L^2}\\
 & \lesssim    2^{ \frac {3q}2}\big\| \Delta_q^\v (R^{\e,3} u^\e_{\rm app} )\big\|_{L^1_{k}  L^2}  d_q  \|  R^\e\| _{\tilde L^\infty_{k}B^{0,\frac12}}\, ,
\end{aligned}
\end{equation}
so we proceed as above by decomposing the first term into   paraproduct and remainder :
 \begin{equation}\label{seconddecompo}
 \begin{aligned}
 2^{ \frac {3q}2}\big\| \Delta_q^\v (R^{\e,3} u^\e_{\rm app} )\big\|_{L^1_{k}  L^2} 
& \lesssim  2^{ \frac {3q}2} \sum_{q'\sim q } \|S_{q'-1}^\v R^{\e,3}\|_{L^\infty_{k}  L^2_\h L^\infty_\v } \|\Delta_{q' }^\v   u^\e_{\rm app}\|_{L^1_{k}  L^\infty_\h L^2_\v} \\
& +  2^{ \frac {3q}2}\sum_{q'\sim q }\|S_{q'-1}^\v  u^\e_{\rm app} \|_{L^2_{k}  L^\infty  } \|\Delta_{q' }^\v R^{\e,3}\|_{L^2_{k}  L^2} \\
& + 2^{ 2q} \sumetage{2^{q'} \gtrsim 2^q} { q'' \sim   {q'}} 
 \|\Delta_{q'' }^\v   u^\e_{\rm app}\|_{L^1_{k}   L^\infty_\h L^2_\v}  \|\Delta_{q' }^\v R^{\e,3}\|_{L^\infty_{k}  L^2} \\
&=:T_{q3}^1 +  T_{q3}^2+ R_{q3}\, .
 \end{aligned}
\end{equation}
Note on the one hand that
   $$
 \begin{aligned}
 T_{q3}^1 
& \lesssim 2^{ \frac {3q}2} \sumetage{2^{q''}\lesssim 2^{q'}}{q'\sim q}  2^\frac {q''}2\|\Delta_{{q''}}^\v R^\e\|_{L^\infty_{k}  L^2} \|\Delta_{q' }^\v     u^\e_{\rm app}\|_{L^1_{k}  L^\infty_\h L^2_\v} \\
 & \lesssim  \| R^\e\| _{\tilde L^\infty_{k}B^{0,\frac12}}  \sum_{q'\sim q } 2^{ \frac {3q'}2} \|  \Delta_{q' }^\v u^\e_{\rm app}\| _{  L^1_{k}  L^\infty_\h L^2_\v}   
 \end{aligned}
$$
and thanks to~(\ref{estimate11}) we therefore have
\beq\label{tq1}
  T_{q3}^1   \lesssim    d_q  \| R^\e\| _{\tilde L^\infty_{k}B^{0,\frac12}}  \|  \partial_3  u^\e_{\rm app}\| _{  L^1_{k}  \mathcal B^{1,\frac12}}  \, .
\eeq
The second term of the decomposition can thus be estimated by~(\ref{inversebernstein}):
$$
 \begin{aligned}
  T_{q3}^2& \lesssim 2^{ \frac {3q}2}\sumetage{2^{q''}\lesssim 2^{q'}}{q'\sim q} 2^\frac {q''}2\|\Delta_{{q''}}^\v u^\e_{\rm app}\|_{L^2_{k}  L^\infty_\h L^2_\v} \|\Delta_{q' }^\v     R^{\e,3} \|_{L^2_{k}L^2} \\
 & \lesssim   2^{ \frac { q}2}\sumetage{2^{q''}\lesssim 2^{q'}}{q'\sim q}2^\frac {q''}2\|\Delta_{{q''}}^\v u^\e_{\rm app}\|_{L^2_{k}  L^\infty_\h L^2_\v} \|\Delta_{q' }^\v     \partial_3 R^{\e,3} \|_{L^2_{k}L^2} \\
 & \lesssim     \|  u^\e_{\rm app}\| _{  \tilde L^2_{k}  \mathcal B^{1,\frac12}}  2^{ \frac { q}2}\sum_{q'\sim q}  \|\Delta_{q' }^\v\nabla^\h R^{\e } \|_{L^2_{k}L^2} 
   \end{aligned}
$$
as above and thanks to~(\ref{estimate11}) and to the fact that $  \partial_3 R^{\e,3} = - \mbox{div}_\h \, R^{\e,\h} $. It follows that
\beq\label{tq2}
  T_{q3}^2  \lesssim   d_q  \|  u^\e_{\rm app}\| _{ \tilde  L^2_{k}  \mathcal B^{1,\frac12}}  \|\nabla^\h R^{\e } \|_{\tilde L^2_{k}B^{0,\frac12}} \, .
\eeq
Finally for the remainder term we write, again by~(\ref{inversebernstein}),
\beq\label{rq}
 \begin{aligned}
R_{q3}
& \lesssim 2^{2q} \sumetage{2^{q'}\gtrsim 2^q}{ q'' \sim   {q'}}
2^ {-q''}\|\Delta_{{q''}}^\v \partial_3 u^\e_{\rm app}\|_{L^1_{k}  L^\infty_\h L^2_\v} \|\Delta_{q' }^\v     R^\e \|_{L^\infty_{k}L^2} \\
& \lesssim \sumetage{2^{q'}\gtrsim 2^q}{ q'' \sim   {q'}} 2^{2(q-q'')}
 2^\frac {q''}2 \|\Delta_{{q''}}^\v \partial_3 u^\e_{\rm app}\|_{L^1_{k}  L^\infty_\h L^2_\v} 2^\frac {q'}2\|\Delta_{q' }^\v     R^\e \|_{L^\infty_{k}L^2} \\
 & \lesssim d_q  \| R^\e\| _{\tilde L^\infty_{k}B^{0,\frac12}}   \| \partial_3 u^\e_{\rm app}\| _{  L^1_{k}  \mathcal B^{1,\frac12}} \, .  \end{aligned}
\eeq
Inserting~(\ref{tq1})-(\ref{rq}) into~(\ref{seconddecompo}) it follows that
$$
 \begin{aligned}
2^{ \frac {3q}2}\big\| \Delta_q^\v (R^{\e,3} u^\e_{\rm app} )\big\|_{L^1_{k}  L^2} 
\lesssim   d_q \Big(     \| R^\e\| _{\tilde L^\infty_{k}B^{0,\frac12}}  \|  \partial_3  u^\e_{\rm app}\| _{  L^1_{k}  \mathcal B^{1,\frac12}}  +  \|\nabla^\h R^{\e } \|_{\tilde L^2_{k}B^{0,\frac12}}  \|  u^\e_{\rm app}\| _{ \tilde  L^2_{k}  \mathcal B^{1,\frac12}}
  \Big)
\end{aligned}
$$
hence returning to~(\ref{estimateJq})
$$2^{ q}\int_{t_k}^{t_{k+1}} J_q(t) dt
  \lesssim s_q  \Big(     \| R^\e\| _{\tilde L^\infty_{k}B^{0,\frac12}}  \|  \partial_3  u^\e_{\rm app}\| _{  L^1_{k}  \mathcal B^{1,\frac12}}  +  \|\nabla^\h R^{\e } \|_{\tilde L^2_{k}B^{0,\frac12}}  \|  u^\e_{\rm app}\| _{ \tilde  L^2_{k}  \mathcal B^{1,\frac12}}
  \Big)
  \|  R^\e\| _{\tilde L^\infty_{k}B^{0,\frac12}}\, .$$
Finally we find \beq\label{estimate2}
 \begin{aligned}
2^{ q}\int_{t_k}^{t_{k+1}}  J_q(t) \, dt&\leq   \frac {s_q}{100}\|\nabla^\h R^\e\|_{\tilde L^2_{k}  B^{0,\frac12}} ^2  \\
 &\quad  + Cs_q
  \| R^\e\|  _{\tilde L^\infty_{k}  B^{0,\frac12}} ^2 \Big( \| u^\e_{\rm app}\| _{\tilde L^2_{k}   {\mathcal B}^{1,\frac12}} ^2+   \|  \partial_3  u^\e_{\rm app}\| _{  L^1_{k}  \mathcal B^{1,\frac12}}\Big)\, .
 \end{aligned}
\eeq
Gathering~(\ref{estimate1}) and~(\ref{estimate2}), Lemma~\ref{demlemproduct} is proved. \qed


 \section{Proof of Lemma~\ref{lemdemforceext}} \label{demforceext}

\subsection{Estimates on the approximate solution} In this section we prove some a priori estimates on~$u^\e_{\rm app}$, whose definition we recall:
$$
 u^\e_{\rm app} =\big[u^\h +\varepsilon w^\h,w^3\big]_\e\, ,
 $$
 with~$u^\h$ solution of~$ {\rm(NS2D)}_{y_3} $ and~$w$ solution of~$ {\rm(T)}$ as defined in the introduction. These estimates were used in Section~\ref{find the proof}
 to prove   Theorem~\ref{thmyotopoulos}.   \begin{prop}\label{propNS2D}
There is a constant~$C_1>0$ depending on~$\|(u_0^\h,w_0)\|_{{\mathcal B}^{0,\frac12}  }$ and~$C_2>0$ depending on~$\|(u_0^\h,w_0)\|_{  {\mathcal B}^{-1,\frac52}\cap {\mathcal B}^{0,\frac12} }$ such that
 $$
\|u^\e_{\rm app} \|_{\tilde L^\infty(\R^+;B^{0,\frac12})} + \| u^\e_{\rm app} \|_{\tilde L^2(\R^+;{\mathcal B}^{1,\frac12})}  +  \|\partial_3 u^\e_{\rm app} \|_{\tilde L^2(\R^+;{\mathcal B}^{0,\frac12})  } + \| u^\e_{\rm app} \|_{\tilde L^2(\R^+;{\mathcal B}^{0,\frac12})}\leq C_1 $$ 
and
$$
 \|  \partial_3 u^\e_{\rm app} \|_{  L^1(\R^+;{\mathcal B}^{1,\frac12})}  \leq \e  \, C_2
 \, .
$$
Finally for any constant~$\bar C>0$ there is a constant~$K>1$ depending on~$\|(u_0^\h,w_0)\|_{{\mathcal B}^{0,\frac12}  }$ and on the times~$0= t_0 <t_1<\dots< t_K=\infty$ such that
$$
\R^+= \bigcup_{k=0}^{K-1} [t_k,t_{k+1}[ \quad  \mbox{and} \quad  \forall \e \in ]0,1[ \, , \,  \,   \|u^\e_{\rm app} \|_{\tilde L^2(\R^+;{\mathcal B}^{1,\frac12})}  (1+\|u^\e_{\rm app} \|_{\tilde L^\infty(\R^+;B^{0,\frac12})} ) \leq \frac1{\bar   C} \, \cdotp
$$
 \end{prop}
 \begin{proof}
By Proposition 3.1 of~\cite{bcg} we know that for any given~$s \in ]-2 ,1[$ and for any~$s'\geq \frac12$ we have 
 \begin{equation}\label{rappel}
\forall r \in [1,\infty] \, , \quad \|u^\h \|_{\tilde L^r(\R^+; {\mathcal B}^{s+\frac2r,s'})}       \lesssim C\, ,
 \end{equation}
 where~$C$ depends on~$\|u_0^\h\|_{{\mathcal B}^{s,s'}\cap {\mathcal B}^{0,\frac12} }$.
On the other hand Proposition 3.5 of~\cite{bcg}  implies that for all~$s \in ]-2 ,0[$ and all~$s'\geq \frac12$
  \begin{equation}\label{rappelw}
 \|w^3 \|_{\tilde L^r(\R^+; {\mathcal B}^{s+\frac2r,s'})}       \lesssim C\, ,
 \end{equation}
where~$C$ depends on the norms of~$u_0^\h$ et~$w_0^3$ in~${\mathcal B}^{s,s'}\cap {\mathcal B}^{0,\frac12} $.
Concerning~$w^\h$ we note that$$
w^\h = - \nabla^\h (\Delta_\h)^{-1} \partial_3 w^3 \, .
$$
As the horizontal average of~$\partial_3w^3$ is zero, 
for all~$s \in ]-2 ,0[$, all~$s'\geq -\frac12$ and all~$r \in [1,\infty]$ there holds\begin{equation}\label{rappelwh}
\begin{aligned}
 \|w^\h \|_{\tilde L^r(\R^+; {\mathcal B}^{s+\frac2r,s'})}  
  &  \lesssim\|\partial_3w^3 \|_{\tilde L^r(\R^+; {\mathcal B}^{s+\frac2r+1,s' })} \\
    &  \lesssim\|\partial_3w^3 \|_{\tilde L^r(\R^+; {\mathcal B}^{s+\frac2r,s' })} \\
 & \leq  C\,,
 \end{aligned}
 \end{equation}
  where~$C$ depends on the norms of~$u_0^\h$ and~$w_0^3$ in~${\mathcal B}^{s,s'+1}\cap {\mathcal B}^{0,\frac12} $. We used above that if a function~$f$ defined on~$\T^2$ has zero horizontal mean, then$$
  s_1 \geq s_2 \Longrightarrow  \|f\|_{{\mathcal B}^{s_1,s'} }\leq \|f\|_{{\mathcal B}^{s_2,s'} }\, .
$$
The first estimate of   Proposition~\ref{propNS2D}
 comes then simply from the fact that by definition of~$ u^\e_{\rm app}$, for any~$\sigma \in\R$,
$$ \| u^\e_{\rm app} \|_{\tilde L^r(\R^+;{\mathcal B}^{\sigma,\frac12})}   =  \|  (u^\h + \e w^\h,w^3) \|_{\tilde L^r(\R^+;{\mathcal B}^{\sigma,\frac12})}    \, ,
$$ 
along with the continuous embedding of~${\mathcal B}^{ \sigma,\frac12}$ into~$ B^{ \sigma,\frac12}$ recalled in~(\ref{injcont}).
 \medskip
 
 \noindent 
For the second estimate of the Proposition~\ref{propNS2D} we apply~(\ref{rappel})-(\ref{rappelwh}) to~$s=-1$, $s'= 3/2$ and~$r=1$. From the definition of~$u^\e_{\rm app}$ we have indeed that

 \begin{equation}\label{estimate00}
\begin{aligned}
   \|  \partial_3 u^\e_{\rm app} \|_{\tilde L^1(\R^+; {\mathcal B}^{1,\frac12})} &= \e   \|  \partial_3 (u^\h + \e w^\h,w^3)\|_{\tilde L^1(\R^+; {\mathcal B}^{1,\frac12})}  
\end{aligned}
\end{equation}
by the same calculations as above,
which completes the proof thanks to~(\ref{rappel})-(\ref{rappelwh}).

 \medskip
 
 \noindent 
Finally, the last result of the proposition is simply that $$
   \|  u^\e_{\rm app} \|_{\tilde L^2(\R^+;{\mathcal B}^{1,\frac12})}  =   \|( u^\h + \e w^\h,w^3)  \|_{\tilde L^2(\R^+;{\mathcal B}^{1,\frac12})}  
$$
and so the time integration interval can be sliced to make the time norms arbitrarily small, regardless of~$\e$. The proposition is proved.
   \end{proof}

   \subsection{Estimates on the pressure}
\begin{prop}\label{propositionpression}
 There are two  constants,~$C_3$ depending on~$\| u_0^\h \|_{{\mathcal B}^{-\frac12,\frac32} \cap{\mathcal B}^{0,\frac12}} $ and~$C_4$ depending on~$\| (u_0^\h,w_0^3) \|_{{\mathcal B}^{-\frac12,\frac52}  \cap{\mathcal B}^{0,\frac12}} $ 
 such that the following holds: $p_0$ satisfies
 $$
 \big \|[\partial_3p_0]_\e\big\|_{L^1(\R^+;  B^{0,\frac12})} 
 \leq C_3 \, ,
$$
and~$p_1$ can be written under the form
$$
p_1 = p_{1,\h} + p_{1,3} 
$$
with
$$
 \big \|[\partial_3p_{1,\h} ]_\e\big\|_{L^1(\R^+;  B^{0,\frac12})} +  \big \|[\nabla^\h p_{1,3} ]_\e\big\|_{L^1(\R^+;  B^{0,\frac12})}  \leq C_4 \, .
$$
\end{prop}
\begin{proof}
By definition
$$
\partial_3p_0 = \partial_3 \sum_{i,j=1}^2 \partial_i \partial_j (-\Delta_\h)^{-1}(u^i u^j)\, .
$$
We recall the product law~(\ref{loiprod})
$$
\|ab\|_{L^1(\R^+; {\mathcal B}^{0,\frac12})} \lesssim \|a \|_{\tilde L^2(\R^+; {\mathcal B}^{\frac12,\frac12})}  \|b \|_{\tilde L^2(\R^+; {\mathcal B}^{\frac12,\frac12})} 
$$
as well as the fact recalled in~(\ref{injcont})
that~${\mathcal B}^{0,\frac12}$ embeds continuously in~$B^{0,\frac12}$. Since the operator~$ \partial_i \partial_j (-\Delta_\h)^{-1}$ is a Fourier multiplier of order 0 if~$i,j \in \{1,2\}$, it follows from~(\ref{rappel}) that
$$
\begin{aligned}
 \big \|[\partial_3p_0 ]_\e\big\|_{L^1(\R^+;  {\mathcal B}^{0,\frac12})}   &= \|\partial_3p_0\|_{L^1(\R^+; {\mathcal B}^{0,\frac12})} \\
 &\lesssim \|\partial_3 u^\h\|_{\tilde L^2(\R^+; {\mathcal B}^{\frac12,\frac12})} \| u^\h\|_{\tilde L^2(\R^+; {\mathcal B}^{\frac12,\frac12})}\\
& \leq C  \,,
\end{aligned}
$$
where~$C$ depends on~$\|u_0^\h\|_{{\mathcal B}^{-\frac12,\frac32}\cap {\mathcal B}^{-\frac12,\frac12}}$.
Furthermore by definition
$$
p_1 = \sum_{i =1} ^2 \sum_{ j=1}^3   \partial_i \partial_j (-\Delta_\h)^{-1}(u^i w^j) 
$$
and one sets
$$
 p_{1,\h} := \sum_{i,j =1} ^2  \partial_i \partial_j (-\Delta_\h)^{-1}(u^i w^j)  \quad  \mbox{and} \quad p_{1,3} 
:= \sum_{i =1} ^2   \partial_i \partial_3 (-\Delta_\h)^{-1}(u^i w^3) \, .
$$
The term~$[\partial_3p_{1,\h} ]_\e$ can be estimated exactly as~$[\partial_3p_0]_\e$ above thanks to~(\ref{rappelw}) and~(\ref{rappelwh}), and similarly for~$[\nabla^\h p_{1,3} ]_\e$ once noticed that for all~$j \in \{1,2\}$,
$$
\partial_j  p_{1,3} =\partial_3\sum_{i =1} ^2  \partial_i\partial_j(u^i w^3)
$$
and using again~(\ref{rappel})-(\ref{rappelwh}).  Proposition~\ref{propositionpression} is proved.
\end{proof}

\subsection{Proof of Lemma~\ref{lemdemforceext}}
recall that
$$  F^\e :=  \e  {\color{black}\Bigl[}\Bigl(w^{\rm h}  \cdot \nabla^{\rm h} (w^{\rm h} ,0 )+  w^3 \partial_3( w^{\rm h},0)\Bigr){\color{black}\Bigr]_\e}+     {\color{black}\Bigl[}\Bigl(w  \cdot\nabla  (u^\h,w^3) \Bigr){\color{black}\Bigr]_\e} +  \big (0,[\partial_3 (p_0 +  \e p_1)]_\e\big)  
  $$
and let us start with the pressure terms, which are estimated by   Proposition~\ref{propositionpression}.   One has indeed on the one hand $$
 \begin{aligned}
2^q  \!  \int_{t_k}^{t_{k+1}}  \! \! \big |  (\Delta_q^\v [\partial_3 (p_0 +  \e p_{1,\h})]_\e  |
 \Delta_q^\v R^\e)_{L^2}(t)  \big |  dt&\leq  C  \| R^\e\|  _{\tilde L^\infty_{k}  B^{0,\frac12}} 
 d_q  2^\frac q 2
\| \Delta_q^\v [\partial_3 (p_0 +  \e p_{1,\h})]_\e\|_{L^1(\R^+; L^2)} \\
 &\leq  Cs_q  \| R^\e\|  _{\tilde L^\infty_{k}  B^{0,\frac12}}  \big \|[\partial_3(p_0 +  \e p_{1,\h}) ]_\e\big\|_{L^1(\R^+;  B^{0,\frac12})}\\
&\leq  Cs_q  \| R^\e\|  _{\tilde L^\infty_{k}  B^{0,\frac12}}  
  \end{aligned}
$$
and on the other hand we notice that
$$
\e     \int_{t_k}^{t_{k+1}}\big |  (\Delta_q^\v (0, [\partial_3    p_{1,3}]_\e)  |
 \Delta_q^\v R^\e)_{L^2}(t) \big | \, dt  =   \int_{t_k}^{t_{k+1}}\big |  (\Delta_q^\v   \partial_3 [   p_{1,3}]_\e   |
 \Delta_q^\v R^{\e,3})_{L^2}\big |  \, dt\, .
$$
After an integration by parts we find therefore $$
\e    \int_{t_k}^{t_{k+1}}\big |  (\Delta_q^\v (0, [\partial_3    p_{1,3}]_\e)  |
 \Delta_q^\v R^\e)_{L^2}(t)\big |  \, dt  =      \int_{t_k}^{t_{k+1}}\big |  (\Delta_q^\v    [   p_{1,3}]_\e   |
 \partial_3\Delta_q^\v R^{\e,3})_{L^2}\big |  \, dt\, .
$$
Since~$R^\e$ is divergence free, another integration by parts gives 
$$
\e    \int_{t_k}^{t_{k+1}}\big |  (\Delta_q^\v (0, [\partial_3    p_{1,3}]_\e)  |
 \Delta_q^\v R^\e)_{L^2}(t) \big | \, dt  =   \int_{t_k}^{t_{k+1}}\big |  (\Delta_q^\v   \nabla^\h [   p_{1,3}]_\e   |
 \Delta_q^\v R^{\e,\h})_{L^2}\big |  \, dt
$$
and we conclude as above thanks to  Proposition~\ref{propositionpression}  that
$$
 \begin{aligned}
\e 2^q   \int_{t_k}^{t_{k+1}}\big |  (\Delta_q^\v (0, [\partial_3    p_{1,3}]_\e)  |
 \Delta_q^\v R^\e)_{L^2} (t)\big | \, dt& =2^q   \int_{t_k}^{t_{k+1}}\big |  (\Delta_q^\v    \nabla^\h [   p_{1,3}]_\e   |
 \Delta_q^\v R^{\e,\h})_{L^2} \big | \, dt \\
  &\leq  Cs_q  \| R^\e\|  _{\tilde L^\infty_{k}  B^{0,\frac12}}  \big \|  \nabla^\h [   p_{1,3}]_\e \big\|_{L^1(\R^+;  B^{0,\frac12})}\\
&\leq  Cs_q  \| R^\e\|  _{\tilde L^\infty_{k}  B^{0,\frac12}}  \, .
  \end{aligned}
$$
It remains to study the bilinear terms. Here again, the product laws recalled in~(\ref{loiprod}) give the result easily since for any function~$a$, thanks to the continuous embedding of~${\mathcal B}^{0,\frac12}$ into~$B^{0,\frac12}$, we have$$
 \begin{aligned}
\|w^{\rm h}  \cdot \nabla^{\rm h} a\|_{L^1(\R^+;  B^{0,\frac12})} &\leq \|w^{\rm h}  \cdot \nabla^{\rm h} a\|_{L^1(\R^+;  {\mathcal B}^{0,\frac12})}\\
& \lesssim  \|w^{\rm h} \|_{\tilde L^2(\R^+;  {\mathcal B}^{1,\frac12})}\|\nabla^\h a \|_{\tilde L^2(\R^+;  {\mathcal B}^{0,\frac12})}
 \end{aligned}
$$
and
$$
 \begin{aligned}
\|w^3   \partial_3 a\|_{L^1(\R^+;  B^{0,\frac12})}  &\leq \|w^3   \partial_3 a\|_{L^1(\R^+;  {\mathcal B}^{0,\frac12})}\\
& \lesssim  \|w^3 \|_{\tilde L^2(\R^+;  {\mathcal B}^{1,\frac12})}\|  a \|_{\tilde L^2(\R^+;  {\mathcal B}^{0,\frac32})}\, .
 \end{aligned}
$$
We conclude thanks to~(\ref{rappel})-(\ref{rappelwh}).   Lemma~\ref{lemdemforceext} is proved. \qed

\appendix

\section{Some   technical tools}

\subsection{Anisotropic Besov spaces} 
 In this Appendix we recall some useful properties on anisotropic Besov spaces, the definition of which is given in the introduction (see Definitions~\ref{debesoviso} and~\ref{deflpanisointro}). 
 
 \medskip
 
  Let us first recall the anisotropic Bernstein inequalities (see~\cite{cz, paicu}).
  
  \smallskip
  
  - If the support of the Fourier transform of a function~$a$ defined on~$\R$ is included in~$2^q \cB$ where~$\cB$ is a ball of~$\R$ then for all~$1\leq p_2\leq p_1\leq
\infty$,
 \beq\label{bernstein}
\|\partial_{x_3}^\alpha a\|_{L^{p_1}(\R)}
\lesssim 2^{q\left(|\al|+ \left(1/{p_2}-1/{p_1}\right)\right)}
\|a\|_{L^{p_2}(\R)}\,.
\eeq

  \smallskip

  - If the support of the Fourier transform of a function~$a$ defined on~$\R$ is included in~$2^q \cC$ where~$\cC$ is a ring of~$\R$ centered at~0 then
 \beq\label{inversebernstein}
\|a\|_{L^{p_1}(\R)} \lesssim
2^{-q} \|\partial_3 a\|_{L^{p_1}(\R)}\,.
\eeq

  \smallskip

  - If the support of the Fourier transform of a function~$a$  defined on~$\T^2$ is included in~$2^j \cB$ where~$\cB$ s a ball of~$\R^2$, then for all~$1\leq p_2\leq p_1\leq\infty$,
 \beq\label{bernstein2D}
\|a\|_{L^{p_1}(\R^2)} \lesssim
2^{2j  \left(1/{p_2}-1/{p_1}\right)} \|  a\|_{L^{p_2}(\R^2)}\,.
\eeq
 It is then not difficult to show, using~(\ref{bernstein2D}), that 
 \beq\label{estimate11}
\sum_{q\in \Z} 2^{qs} \|  \nabla^\h\Delta_{q  }^\v a\| _{  L^2}  + \sum_{q\in \Z} 2^{qs} \|  \Delta_{q  }^\v a\| _{    L^\infty_\h L^2_\v}  \lesssim \|a\|_{{\mathcal B}^{1,s}}\, .
\eeq
The spaces given by the following norm, introduced in~\cite{cheminlerner}, are used consistently in this text:
$$
\begin{aligned}
\|a\|_{\tilde L^r( [0;T];{\mathcal B}^{\sigma,s})} &:=\sum_{j,q} 2^{js + qs'}\|\Delta_j^{\rm h} \Delta_{q}^{\rm v} f\|_{ L^r( [0;T];L^2)} \\
\|a\|_{\tilde L^r( [0;T];{ B}^{0,s})}&:=\sum_{q} 2^{qs }\|\Delta_j^{\rm h} \Delta_{q}^{\rm v} f\|_{ L^r( [0;T];L^2)}\, .
\end{aligned}
$$
 Finally, let us present the paraproduct algorithm of J.-M. Bony~\cite{Bo81} (in the vertical direction): the product of two distributions~$a,b$, when defined, can decompose into
$$
ab =  S_{q -1}^\v a \, \Delta_{q  }^\v   b+ S_{q -1}^\v  b  \, \Delta_{q }^\v a+ \sum_{ q \sim   {q'}} 
 \Delta_{q}^\v   a \, \Delta_{q' }^\v b$$
and thus in particular \beq\label{paraproduit}
 \Delta_{q }^\v (ab) = \sum_{q'\sim q}  S_{q'-1}^\v a \, \Delta_{q' }^\v  b+ \sum_{q'\sim q } S_{q'-1}^\v  b  \, \Delta_{q' }^\v a+ \sumetage{2^{q'} \gtrsim 2^q} { q'' \sim   {q'}} 
 \Delta_{q'' }^\v   a \, \Delta_{q' }^\v b\, .
\eeq
This decomposition, with~(\ref{bernstein}), makes it possible to prove the following product laws (see for example~\cite{bcg}):
\begin{equation}\label{loiprod}
\begin{aligned}
\forall s \geq 1/2\, , \quad  \|ab\|_{ {\mathcal B}^{1,s}}& \lesssim \|a \|_{{\mathcal B}^{1,s}}  \|b \|_{ {\mathcal B}^{1,s}} \\
\|ab\|_{ {\mathcal B}^{0,s}}& \lesssim \|a \|_{{\mathcal B}^{\frac12,s}}  \|b \|_{ {\mathcal B}^{\frac12,s}} \\
\|ab\|_{ {\mathcal B}^{0,s} }&  \lesssim \|a \|_{{\mathcal B}^{1,s}}  \|b \|_{ {\mathcal B}^{0,s}} \, .
\end{aligned}
\end{equation}

\subsection{Some useful trilinear estimates}
We recall here for the convenience of the reader some estimates which were used in the course of the proofs.

\medskip

First from \cite[Section 4.1, Corollary 3]{paicu} we have  for any divergence free vector field~$u$  
\begin{equation}\label{estimRRappendix}
 \int_{a}^{t_{k+1}} \big | ( \Delta_q^\v (u\cdot \nabla u ) |
 \Delta_q^\v u)_{L^2}(t)
\big | 
 dt\lesssim 2^{-q}s_q \|\nabla^\h u\|_{\tilde L^2_{k} B^{0,\frac12}}^2 \| u\|_{\tilde L^\infty_{k} B^{0,\frac12}} 
\end{equation}
and from \cite[Lemma from Section 5.1]{paicu}, for all $v$ and for~$u$ divergence free
\begin{equation}\label{estimRuappendix}
\begin{aligned}
 & \int_{t_k}^{t_{k+1}}  \big | ( \Delta_q^\v (u\cdot \nabla v) |
 \Delta_q^\v R^\e)_{L^2}(t)
\big | 
 dt\lesssim   2^{-q}s_q \| v\|^\frac12 _{\tilde L^\infty_{k}  B^{0,\frac12}}\|\nabla^\h v\|_{\tilde L^2_{k}  B^{0,\frac12}} \\
 &\quad \times \Big(\|\nabla^\h v\|^\frac12_{\tilde L^2_{k}  B^{0,\frac12}}
  \| u\|^\frac12 _{\tilde L^\infty_{k}  B^{0,\frac12}}\|\nabla^\h u\|^\frac12_{\tilde L^2_{k}  B^{0,\frac12}} 
   +\|\nabla^\h u\| _{\tilde L^2_{k}  B^{0,\frac12}} 
 \| v\|^\frac12 _{\tilde L^\infty_{k}  B^{0,\frac12}}\Big)
\end{aligned}
\end{equation}
where~$(s_q)_{q \in \ZZ}$ is any sequence of positive real numbers satisfying
$$
\sum_{q \in \ZZ} s_q^\frac12 \leq 1 \, .
$$

\end{document}